\documentclass[12pt]{amsart}

\title[Sierpi\'nski dynamical systems]{On locally distinguishing Sierpi\'nski dynamical systems}

\author[S. Merenkov]{Sergei Merenkov}
\address{Department of Mathematics, The City College of New York and CUNY Graduate Center, New York, NY 10031, USA}
\email{smerenkov@ccny.cuny.edu}
\thanks{Supported by NSF grant DMS-2247364.}
\subjclass[2020]{}

\usepackage{amsmath, amssymb, amsthm,latexsym}
\usepackage{graphicx}
\usepackage{amscd}
\newcommand\C{{\mathbb C}}

\newcommand\N{{\mathbb N}}

\newcommand\Z{{\mathbb Z}}

\newcommand\co{\colon}
\renewcommand\:{\colon}
\newcommand\sub {\subseteq}

\newcommand\wh{\widehat}
\newcommand\wt{\widetilde}

\newtheorem{theorem}{Theorem}[section]
\newtheorem{definition}[theorem]{Definition}

\newtheorem{corollary}[theorem]{Corollary}

\theoremstyle{definition}

\usepackage{hyperref}
\hypersetup{
    colorlinks=true,
    linkcolor=blue,
    filecolor=magenta,      
    urlcolor=cyan,
}

\begin{document}


\abstract{We prove that there is no local quasiconformal map between the limit set of a convex-cocompact Kleinian group and the Julia set of a postcritically finite rational map, provided that both are Sierpi\'nski carpets. This contrasts with the recent results by Y.~Luo, D.~Ntalampekos and Y.~Luo, M.~Mj, S.~Mukherjee for gasket and tree-like spaces, respectively.  
}
\endabstract

\maketitle


\section{Introduction}\label{s:Intro}

\noindent
Recently, Y.~Luo and D.~Ntalampekos proved the following result.
\begin{theorem}
	\cite[Theorem~1.5]{LN24}
	There exist a
	gasket limit set $\Lambda(G)$ of a Kleinian group $G$ and a gasket Julia set $\mathcal J(f)$ of a rational map $f$ such that for any point $z \in \Lambda(G)$, there exist a neighborhood $U$ of $z$
	and a quasiconformal map $\phi\: \wh\C \to\wh\C$ with $\phi(U\cap \Lambda(G)) = \phi(U)\cap \mathcal J(f)$, where $\wh\C$ denotes the Riemann sphere.
\end{theorem}

Here, a set $K\subset\wh\C$ is a \emph{gasket} if it is a locally connected continuum
with empty interior and that has the following properties: each complementary component of $K$ is a Jordan region, the boundaries of any two complementary components of $K$ share at most
one point, no point of $\wh\C$ belongs to the boundaries of three complementary components of $K$, the \emph{contact graph} corresponding to $K$, obtained by assigning a vertex to each complementary component and an edge if two components share a boundary point, is connected. For the definition of quasiconformality and related concepts please see Section~\ref{S:Prelim}.
\begin{corollary}\cite[Corollary~1.6]{LN24}\label{C:GLS}
	There exist a gasket Kleinian limit set
	$\Lambda(G)$, open sets $U, V \subset \wh\C$ such that  $U\cap \Lambda(G),V \cap \Lambda(G)$ are non-empty connected sets, and a quasiregular map $f \: U \to V$ of degree greater than one with $f(U \cap \Lambda(G)) = V \cap \Lambda(G)$.
\end{corollary}

In fact, in the construction in \cite{LN24} the subset $A = U\cap\Lambda(G)$ decomposes into 4 connected pieces $A_0,\dots A_3$, where $A_i\cap A_{i+1}$, for every $i\in\Z \mod 4$, is a single point and $f|_{A_i}$ is a M\"obius map.

In a similar direction, Y.~Luo, M.~Mj, and S.~Mukherjee~\cite{LMM26} showed that the fat Basilica Julia set of $f(z)=z^2-3/4$ is quasiconformally equivalent, via a global quasiconformal homeomorphism of the plane, to the limit set of a Kleinian group.

A \emph{Sierpi\'nski carpet}, or \emph{carpet} for short, is a metrizable to\-po\-lo\-gi\-cal space homeomorphic to the standard Sierpi\'nski carpet (aka the Sierpi\'nski curve) obtained by subdividing the unit square into nine equal subsquares, removing the interior of the middle subsquare, performing the above operations on the remaining eight subsquares, and continuing indefinitely. Similar to a gasket, it is a compact, connected, locally connected space of topological dimension one, but, unlike a gasket, the closures of any two complementary components of a carpet do not intersect. In particular, there is no contact graph, i.e., no underlying combinatorial structure, for a carpet, and there is only one to\-po\-lo\-gi\-cal carpet type.    
For global quasiconformal maps between Sierpi\'nski carpet Julia sets of postcritically finite rational maps M.~Bonk, M.~Lyubich and the author~\cite{BLM16} proved the following theorem.
\begin{theorem}\cite[Theorem~1.4]{BLM16}
	If $f, g$ are postcritically finite  rational  maps whose respective Julia sets $\mathcal J(f), \mathcal J(g)$ are Sierpi\'nski carpets, then any quasiconformal map $\xi\co\wh\C\to\wh\C$ such that $\xi(\mathcal J(f))=\mathcal J(g)$ is the restriction to $\mathcal J(f)$ of a M\"obius transformation.
\end{theorem}

Since a Kleinian group acts on its limit set, the following is an immediate corollary.

\begin{corollary}\cite[Corollary~1.3]{BLM16}\label{C:SCGJ}
	Let $f$ be a postcritically finite rational map whose Julia set $\mathcal J(f)$ is a Sierpi\'nski carpet. Then $\mathcal J(f)$ is not quasiconformally equivalent to the limit set $\Lambda(G)$ of a convex-cocompact Kleinian group $G$ via a global quasiconformal map of the Riemann sphere.
\end{corollary}

The following local result proved by the author~\cite{Me14a} contrasts Corollary~\ref{C:GLS} when limit sets are Sierpi\'nski carpets rather than gaskets. 

\begin{theorem}\cite[Theorem~1.1]{Me14a}\label{T:Groups}
	Let $G, \widetilde G$ be convex-cocompact Kleinian groups
	with $\Lambda(G), \Lambda(\widetilde G)$ being round Sierpi\'nski carpets, i.e., carpets in $\wh\C$ whose complementary components are open geometric disks. If $U$ and $V$ are open sets such that $U\cap \Lambda(G)$ and $V\cap\Lambda(\widetilde G)$ are non-empty and connected, and if
	$f\colon U\cap \Lambda(G) \to V\cap\Lambda(\widetilde G)$ is a quasiregular map, then there exists a M\"obius transformation
	$\mu$ that takes $\Lambda(G)$ onto $\Lambda(\widetilde G)$, and such that
	$$\mu|_{U\cap\Lambda(G)}=f.$$
	\end{theorem}
	
In other words, round carpet limit sets of convex-cocompact Kleinian groups support no branched quasiregular maps, even locally. In this note we establish a local analogue of Corollary~\ref{C:SCGJ} by proving the following theorem.

\begin{theorem}\label{T:Main}
	Let $G$ be a convex-cocompact Kleinian group and $f$ be a postcritically finite rational map. Assume that the limit set $\Lambda(G)$ of $G$ and  the Julia set $\mathcal J(f)$ of $f$ are Sierpi\'nski carpets. Then there does not exist a quasiconformal map $\xi\co U\cap \Lambda(G)\to V\cap \mathcal J(f)$, where $U, V$ are open sets such that the intersections $U\cap \Lambda(G)$ and $V\cap \mathcal J(f)$ are non-empty and connected.
\end{theorem}

In Section~\ref{S:Prelim} we discuss relevant classes of deformations and pre\-li\-mi\-na\-ry results in order to give a proof of Theorem~\ref{T:Main} in Section~\ref{S:Proof}. 

\medskip
\noindent
{\bf Acknowledgments.} 
The author thanks Yusheng Luo for useful discussions related to the project.
The author also thanks the Institute for Ma\-the\-ma\-ti\-cal Sciences at Stony Brook University for its hospitality.

\section{Preliminaries}\label{S:Prelim}

\noindent
Recall that a \emph{Sierpi\'nski carpet} $C$ is a topological space homeomorphic to the classical Sierpi\'nski curve $S$. The images of the boundaries of the removed squares in the construction of $S$ under any such homeomorphism are called \emph{pe\-ri\-phe\-ral circles}.
A \emph{Schottky set} $S$ in $\wh\C$ is a complement in $D$ of a collection of open geometric disks $D_i,\ i\in I$, with disjoint closures, i.e., 
$$
S=\wh\C\setminus \cup_{i\in I}D_i,
$$
with $\overline{D_i}\cap{\overline D_j}=\emptyset,\ i,j\in I,\  i\neq j$. The boundaries $\partial D_i,\ i\in I$, are also referred to as the \emph{peripheral circles} of $S$. 

A non-constant map $f\co U\to \wh\C$ defined on a domain $U\sub \wh\C$ is called \emph{$K$-quasiregular}, $K\geq 1$, if $f$ is in the Sobolev space $W_\mathrm{loc}^{1,2}(U)$, and for almost every $z\in U$,
$$
\left\lVert Df(z)\right\rVert^2\leq K\mathrm{det}(Df(z)).
$$
A map $f$ is \emph{quasiregular} if it is  $K$-quasiregular for some $K\geq 1$. 
If a $K$-quasiregular map $f$ is a homeomorphism onto its image, then $f$ is called $K$-\emph{quasiconformal}, or just \emph{quasiconformal}. A 1-quasiconformal map is {conformal} by Weyl's Lemma.

The following theorem due to M.~Bonk, B.~Kleiner and the author~\cite{BKM09} will be used in the proof of Theorem~\ref{T:Main} below.
\begin{theorem}\cite[Theorem~1.1]{BKM09}\label{T:BKM}
	If $S$ is a Schottky set of measure zero, then for any quasiconformal map $f$ of $\wh\C$ such that the image $f(S)$ is a Schottky set, its restriction to $S$ is a M\"obius transformation.
	\end{theorem}

The following theorem by M.~Bonk~\cite{Bo11} allows one to conjugate certain expanding conformal or rational dynamical systems on Sierpi\'nski carpets to those of uniformly quasiconformal or quasiregular systems on Schottky sets. 

\begin{theorem}\cite[Theorem~1.1]{Bo11}\label{T:Bo}
	Let $C$ be a Sierpi\'nski carpet in $\wh\C$ whose peripheral circles are uniformly relatively separated uniform quasicircles. Then there exists $\beta$, a quasiconformal map of $\wh\C$, such that $S=\beta(C)$ is a Schottky set.
\end{theorem}
\noindent
Here, peripheral circles $\{C_i,\ i\in I\}$ of $C$ being \emph{uniform quasicircles} means that there exists $K\ge1$ such that each peripheral circle of $C$ is an image of the unit circle under a global $K$-quasiconformal map of $\wh\C$. A family of continua $\{C_i,\ i\in I\}$ in $\wh\C$ is said to be \emph{uniformly relatively separated} if there exists $\delta>0$ such that the relative distance
$$
\Delta(C_i, C_j)=\frac{{\rm dist}(C_i, C_j)}{\min\{{\rm diam}(C_i),{\rm diam}(C_j)\}}\ge\delta,\quad i,j\in I,\ i\neq j, 
$$ 
where ${\rm dist}$ stands for the spherical distance and ${\rm diam}$ the spherical diameter. Note that for a family of circles in $\wh\C$ being uniform quasicircles as well as being uniformly relatively separated are invariant properties under global quasiconformal maps of $\wh\C$.  

%

An earlier result by the author~\cite[Theorem~1.2]{Me12}, combined with \cite[Lemma~2.1]{Me14a}, guarantees that locally defined quasiconformal maps between Schottky sets enjoy differentiability properties akin to those of conformal maps.

\begin{theorem}\cite[Theorem~1.2]{Me12}\label{T:SchM}
Let $S, \widetilde S$ be Schottky sets, with $S$ having measure zero.
Let $f\co U\to\wt U$ be a quasiconformal map between open sets $U,\wt U$, such that $f(U\cap S)=\wt U\cap\widetilde S$. Then  
$f$ is differentiable along $S$ in the sense that for each $p\in U\cap S$, the derivative along $S$ defined by
\begin{equation}\label{E:Conf}
f'(p)=\lim_{q\in S,\, q\to p}\frac{f(q)-f(p)}{q-p}
\end{equation}
exists. 
Additionally, the derivative $f'$ is continuous along $S$.
\end{theorem}

This theorem motivated the following definition~\cite{Me14b}, extended to non-local homeomorphisms.

\begin{definition}
	If $S$ and $\wt S$ are Schottky sets, a continuous map $f\co U\cap S\to\widetilde S$ is called a Schottky map if $f^{-1}(\wt S)=U\cap S$, there exists a discrete subset $E$ of $U\cap S$ such that the map $f$ is differentiable along $S$ in $(U\cap S)\setminus E$ in the sense of~\eqref{E:Conf}, and the derivative $f'$ is continuous along $S$ in $(U\cap S)\setminus E$. 
\end{definition}

The reason we extend the notion of a Schottky map to non-local homeomorphisms, i.e., we allow for an exceptional discrete set $E$, is to include quasiregular, rather than just quasiconformal, maps; see also~\cite{MS26}. Quasiregular maps are not local homeomorphisms near a discrete set of branch points.
%
%
%
%
%
%
It turns out that Schottky maps posses a great deal of rigidity, even beyond that of conformal maps, as the following results from~\cite{Me14b} indicate.
\begin{theorem}\cite[Theorem~4.1]{Me14b}\label{T:Rig}
	Let $S$ be a locally porous Schottky set and let $U\subseteq\C$ (here we view $\wh\C=\C\cup\{\infty\}$) be an open set with $U\cap S$ connected. Let $f\co U\cap S\to S$ be a Schottky map and $p\in U\cap S$. Assume further that $f(p) = p$ and $f'(p) = 1$. Then $f = id$ in $U\cap S$.	
\end{theorem}
\noindent
A Schottky set $S$ is called \emph{locally porous} at $p\in S$ if there exist a neighborhood $U$ of $p$, a constant $r_0>0$ and a constant $\alpha\ge1$, such that for every $q\in U\cap S$ and each $r$ with $0 <r\le r_0$, there exists a peripheral circle $C_i,\ i\in I$, of $S$ such that $B(q,r)\cap C_i \neq\emptyset$ and 
$$r/\alpha \le {\rm rad}(C_i)\le \alpha r,$$ where ${\rm rad}$ stands for the radius of $C_i$. 
A Schottky set $S$ is called locally porous if it is locally porous at every $p\in S$. Every locally porous Schottky set has measure zero since it cannot have Lebesgue density points. In addition, the local porosity property is invariant under global quasiconformal maps of $\wh\C$. 
Since the idea of the proof of Theorem~\ref{T:Rig} is short and elementary, we include it here for reader's convenience; please see~\cite[Theorem~4.1]{Me14b} for more details if necessary.
\begin{proof}
Connectedness of $U\cap S$ implies that it is enough to assume for contradiction that there exists a sequence $(p_k)_{k\in\N},\ p_k\in U\cap S$, such that $p_k\to p$ as $k\to\infty$, but $f(p_k)\neq p_k$ for all $k\in\N$.
For each $k\in\N$, we define $\lambda_k=f(p_k)-p_k$; so the sequence $(\lambda_k)_{k\in\N}$ consists of non-zero numbers and its limit is 0. If we denote by $S_k, S_k(U)$ the preimages of the sets $S$ and $U\cap S$, respectively, by the linear map $\phi_k(z)=p_k+\lambda_k z$, then we have the following commutative diagram
$$
\begin{CD}
	S_k(U) @>f_k>> S_k\\
	@V{\phi_k}VV
	@VV{\phi_k}V\\
	U\cap S
	@>f>>
	S,
\end{CD}
$$ 
where 
$$
f_k(z)=\frac{f(p_k+\lambda_k z)-p_k}{\lambda_k}=\frac{f(p_k+\lambda_k z)-f(p_k)}{\lambda_k z}z+1,\quad k\in\N.
$$
Since $f'$ is assumed to be continuous and $f'(p)=1$, the sequence $(f_k)_{k\in\N}$ converges to the translation $t(z)=z+1$ uniformly on compacta in $\C$ intersected with $S_k(U),\ k\in\N$. Note that both sequences of sets $(S_k)_{k\in\N}, (S_k(U))_{k\in\N}$ have subsequences, and therefore a common subsequence, that converge in the Hausdorff topology to the same, so-called \emph{weak tangent space} $S_\infty\subset\C$ of $S$ at $p$.
The local porosity assumption of $S$ implies that the complement of $S_\infty$ in $\C$ contains disks of arbitrarily large radii. However, from the above commutative diagram we conclude that the set $S_\infty$ has to be invariant under the translation $t$.   
This is a contradiction. 
\end{proof}

The following is an immediate consequence of Theorem~\ref{T:Rig}.
\begin{corollary}\cite[Corollary 4.2]{Me14b}\label{C:Uni}
	 Let $S$ be a locally porous Schottky set and suppose that $U$ is an open set such that $U\cap S$ is connected. Let $f, g\colon U\cap S \to \wt S$ be Schottky maps into a Schottky set $\wt S$, and consider $F =\{p\in U\cap S\colon f(p)=g(p)\}$. Then $F =U\cap S$, provided $F$ has an accumulation point in $U$. 
\end{corollary}

\section{Proof of Theorem~\ref{T:Main}}\label{S:Proof}		
		
\noindent
According to~\cite[Proposition~1.4]{Bo11} and~\cite[Theorem~1.10]{BLM16}, the the limit set $\Lambda(G)$ of a convex-cocompact Kleinian group $G$ and the Julia set $\mathcal J(f)$ of a postcritically finite rational map $f$ are locally porous sets whose peripheral circles are uniformly relatively separated uniform quasicircles. Therefore, by Theorem~\ref{T:Bo}, we can conjugate both dynamical systems, namely the group $G$ acting on its limit set and the rational map $f$ on its Julia set, by the corresponding quasiconformal maps $\beta_G$ and $\beta_f$ to dynamical systems acting on Schottky sets. 
Note that, even though $\beta_G$ is only quasiconformal, the elements of the conjugate group $\beta_G\circ G\circ\beta_G^{-1}$ are the restrictions of M\"obius transformations by Theorem~\ref{T:BKM}. Note here that the local porosity property implies that the limit and Julia sets have zero measure. Thus, the given Kleinian group $G$ restricted to $\Lambda(G)$ conjugates by $\beta_G$ to a Kleinian group, necessarily convex-cocompact. The corresponding conjugate map $\beta_f\circ f\circ\beta_f^{-1}$ is uniformly quasiregular in the sense that all the iterates of $f$ are $K$-quasiregular for the same constant $K\ge1$. Moreover, the conjugate map $\beta_f\circ f\circ\beta_f^{-1}$ is a Schottky map on $\beta_f(\mathcal J(f))$ by Theorem~\ref{T:SchM}. 
To simplify notations, from now on we denote the Schottky sets $\beta_G(\Lambda(G))$ and $\beta_f(\mathcal J(f))$ by $\Lambda$ and $\mathcal J$, respectively, and keep the same notations $G$ and $f$ for the conjugate group $\beta_G\circ G\circ\beta_G^{-1}$ and the conjugate map $\beta_f\circ f\circ\beta^{-1}$, respectively. E.g., now $f$ is a uniformly quasiregular map that is Schottky on $\mathcal J$. We also use the same notation $\xi$ for the map $\beta_f\circ\xi\circ\beta_G^{-1}$, where $\xi$ is the hypothetical quasiconformal map from the statement, so that  $\xi\co U\cap\Lambda\to V\cap\mathcal J$ is a Schottky map by Theorem~\ref{T:SchM}.
Since the local porosity property is preserved under quasiconformal maps, the Schottky sets $\Lambda$ and $\mathcal J$ are locally porous.

Let $\tilde p\in V\cap\mathcal J$ be a repelling periodic point of $f$. Such points are dense in $\mathcal J$. By possibly passing to an iterate of $f$, we may and will assume that $\tilde p$ is fixed by $f$, and let $p=\xi^{-1}(\tilde p)\in U\cap \Lambda$. This is a fixed point of the map $f_\xi=\xi^{-1}\circ f\circ\xi$ defined in an open set $U'\subset U$, which we can choose so that $U'\cap\Lambda$ is connected. By Theorem~\ref{T:Groups}, the map $f_\xi$ is the restriction to $U'\cap\Lambda$ of a M\"obius transformation $\mu$ preserving $\Lambda$. Equivalently, 
\begin{equation}\label{E:ToExt}
\xi=f\circ\xi\circ \mu^{-1}	
\end{equation}
in $\mu(U')\cap\Lambda$. 
Since $p$ was chosen so that $\tilde p=\xi(p)$ is a repelling fixed point of $f$, the map $\mu$ is necessarily loxodromic with pole at $p$. Iterating Equation~\eqref{E:ToExt}, namely multiplying it by $f$ on the left and $\mu^{-1}$ on the right over and over, we conclude that $\xi$ extends as a Schottky map to $\Lambda\setminus\{q\}$, where $q$ is the attracting, i.e., other than $p$, fixed point of $\mu$.  Since $\mu$ preserves $\Lambda$, one necessarily has $q\in\Lambda$. Equation~\eqref{E:ToExt} along with the assumption that $\xi$ is quasiconformal in $U$ implies, in particular, that $\xi$ has a quasiregular extension to the complement of any closed disk centered at $q$. Indeed, if $r>0$ is arbitrary and $B(q,r)$ is an open disk centered at $q$ of radius $r$, we can find a sufficiently large $n\in\N$ such that $\mu^{-n}(\wh\C\setminus \overline{B(q,r)})\subset\mu(U')$. Since $\xi$ is quasiconformal in $U$, the map $f$ is quasiregular, and $\mu$ is a M\"obius transformation, the formula $\xi=f^n\circ\xi\circ \mu^{-n}$ gives such an extension to $\wh\C\setminus \overline{B(q,\delta)}$. In particular, the Schottky  map $\xi|_{\Lambda\setminus\{q\}}$ has at most finitely many critical points in a neighborhood of any point of $\Lambda$ other than, possibly, $q$.

If $\xi$ had a finite degree extension to $\Lambda$, Equation~\eqref{E:ToExt} would give a contradiction via the degree count. To conclude a finite degree extension for $\xi$, it is enough to show that $\xi$ has a quasiregular extension in a neighborhood of $q$.  To this end, let $\wt{p'}\in V$ be another repelling fixed point of an iterate of $f$, which we also assume to be $f$ itself, and let $\mu'=\xi^{-1}\circ f\circ\xi$. As above, $\mu'$ extends as a M\"obius transformation preserving $\Lambda$, with a repelling pole at $p'=\xi^{-1}(\wt{p'})$.
Let $q'$ be the attracting pole of $\mu'$.
Using the equation
\begin{equation}\label{E:ToExt2}
	\xi=f\circ\xi\circ \mu'^{-1},	
\end{equation}
we can extend $\xi$ to $\Lambda\setminus\{q'\}$. As above, the extension is necessarily a Schottky map. Moreover, Equation~\eqref{E:ToExt2} gives a quasiregular extension of $\xi$ to $\wh\C\setminus\overline{B(q',r')}$, for arbitrary $r'>0$. By Corollary~\ref{C:Uni}, the two extensions of $\xi$ agree on $\Lambda\setminus\{q,q'\}$. If $q\neq q'$, we conclude that $\xi$ has a quasiregular extension to every point of $\wh\C$, and either of the equations~\eqref{E:ToExt}, \eqref{E:ToExt2} gives a contradiction via the degree count on $\Lambda$. 

Assume that $q=q'$. In this case $\mu$ and $\mu'$ commute. Indeed, the commutator $\mu\circ\mu'\circ\mu^{-1}\circ\mu'^{-1}$ fixes $q=q'$ and has derivative equal to one at this point. Since this is trivially a Schottky map preserving $\Lambda$, Theorem~\ref{T:Rig} gives that the commutator is the identity, i.e., $\mu$ and $\mu'$ commute.
We now consider $\mu''=\mu'^{-1}\circ\mu$, a M\"obius transformation preserving $\Lambda$ that has a repelling fixed point at $p$ and an attracting fixed point at $p'\neq q$. Therefore, we can use $\mu''$ instead of $\mu'$ in the previous argument to derive a contradiction.   
\qed

\end{document}